\documentclass[final,hidelinks]{siamltex704}
\usepackage{amssymb,amsmath,graphicx,amscd,mathrsfs}
\usepackage{enumerate}
\usepackage{color,xcolor,amsmath}
\usepackage{verbatim}
\usepackage{amsmath}
\usepackage{graphicx}
\usepackage[notcite,notref]{showkeys}
\usepackage{mathrsfs}
\usepackage{float}
\usepackage{amsfonts,amssymb}
\usepackage{dsfont}
\usepackage{pifont}
\usepackage{hyperref}
\usepackage{multirow}
\usepackage{tikz}
\usepackage{subfigure}
\numberwithin{equation}{section}
\def\3bar{{|\hspace{-.02in}|\hspace{-.02in}|}}

\def\T{{\mathcal{T}}}

\def\b0{\boldsymbol{0}}

\def\bn{{\mathbf{n}}}

\def\bq{{\mathbf{q}}}

\def\b0{{\mathbf{0}}}

\def\bP{{\mathbf{P}}}
\def\ba{{\mathbf{a}}}
\def\bb{{\mathbf{b}}}

\newtheorem{example}{\bf Example}[section]
\newtheorem{remark}{Remark}[section]

\title{Supercloseness Analysis and Polynomial Preserving Recovery for a Class of Weak Galerkin Methods}

\author{
Ruishu Wang\thanks{Department of Mathematics, Jilin University, Changchun 130012,
China ({\tt wangrs16@mails.jlu.edu.cn}).}
\and
Ran Zhang\thanks{Department
of Mathematics, Jilin University, Changchun 130012, China
({\tt zhangran@mail.jlu.edu.cn}). The research of this author was supported in
part by China Natural National Science Foundation grants 11271157,
11371171, 11471141, U1530116, and by the Program for New Century
Excellent Talents in University of Ministry of Education of China.}
\and
Xu Zhang\thanks{Department of Mathematics and Statistics, Mississippi State University, Mississippi State, MS 39762, USA
({\tt xuzhang@math.msstate.edu})}
\and
Zhimin Zhang\thanks{Beijing Computational Science Research Center, Beijing 100193, China ({\tt zmzhang@csrc.ac.cn});
Department of Mathematics, Wayne State University, Detroit, MI 48202, USA ({\tt ag7761@wayne.edu}).
 The research of this author was supported in part by the US National Science Foundation grant DMS-1401090 and the
 National Natural Science Foundation of China grants 11471031, 91430216, U1530401.} }

\begin{document}

\maketitle

\begin{abstract}
In this paper, we analyze the convergence and supercloseness properties of a class of weak Galerkin (WG) finite element methods for solving second-order elliptic problems. It is shown that the WG solution is superclose to the Lagrange type interpolation using Lobatto points. This supercloseness behavior is obtained through some newly designed stabilization terms. A post-processing technique using the polynomial preserving recovery (PPR) is introduced for the WG approximation. Superconvergence analysis is carried out for the PPR approximation. Numerical examples are provided to verify our theoretical results.
\end{abstract}



\section{Introduction}
The weak Galerkin (WG) finite element methods (FEM) refer to a new class of
finite element discretizations for solving partial differential equations (PDE). In the WG methods, classical differential operators are replaced by generalized differential operators as distributions. Unlike the classical FEM that impose continuity in the approximation space, the WG methods enforce the continuity weakly in the formulation using generalized weak derivatives and parameter-free stabilizers. The WG methods are naturally extended from the standard FEM, and are more advantageous over FEM in several aspects. For instance, high order WG spaces are usually more convenient to construct than conforming FEM spaces since there is no continuity requirement on the approximation spaces. Also, the relaxation of the continuity requirement enables easy implementation of WG methods on polygonal meshes.

The first WG method was introduced in \cite{WangYe13} for the second-order elliptic equation, in which the $H(div)$ finite elements such as Raviart Thomas elements are used to approximate weak gradients. Later in \cite{MuWangYe15, WangYe14}, WG methods following the stabilization approach were introduced, which can be applied on polygonal meshes. This new stabilized WG discretization has been applied to many classical PDE models, such as elliptic interface problems \cite{MWYinterface}, the Maxwell equation \cite{mwyz-maxwell}, Brinkman equation \cite{MWY_JCP14,zzm16}, and biharmonic equation \cite{mwyz-biharmonic,zz-biharmonic}.

It is well known that superconvergence is an important and desirable mathematical property of numerical methods for solving PDE. Superconvergence phenomenon means the convergence rate at certain points is higher than the optimal global convergence rate of numerical solutions. Due to its wide application, superconvergence has been extensively studied in the past decades, see for example \cite{A84,ADT81,babuska,2015CaoZhangZhang,D72,HLL12,krizek,SSW96,W95}.
There are also some literature on superconvergence analysis for WG methods. For instance, in \cite{WangYe13}, the error estimate revealed a superconvergence for the WG approximation (without stabilization terms) on simplicial meshes. In \cite{HH14}, superconvergence of the WG methods with stabilizers are obtained by $L^2$ projection methods.

One goal of this article is to analyze the supercloseness property of a class of WG methods with generalized stabilizers.  Unlike the stabilizer introduced in \cite{MuWangYe15}, there is a fine-tune parameter in our new stabilizer \eqref{bi_form_1}, and it reduces to the standard stabilizer when the parameter $\alpha=1$.  We will show that this new parameter plays a critical role in the analysis for supercloseness. To be more specific, we show that the new WG solutions are superclose to a Lagrange type interpolation of the exact solution.

Another focus of this article is to develop an efficient post-processing technique of WG methods which leads to a better approximation of the gradient of solution. We adopt the polynomial preserving recovery (PPR) technique \cite{Naga2004,Zhang2005,ZN05} in our post-processing. The main idea of PPR is to construct a higher-order polynomial locally around each node based on current numerical solution. Unlike the standard FEM approximation which is a continuous function, WG solution is discontinuous across the boundary of elements; hence, there can be multiple values associated with a single node. Consequently, we will need to introduce an appropriate weighted average to unify these values before applying the standard PPR scheme. The analysis of superconvergence of PPR scheme relies heavily on the aforementioned supercloseness property.

The rest of the paper is organized as follows. In Section 2, we introduce the definition of weak functions/derivatives, and present the WG method for the model second order elliptic equation.  In Section 3, we describe a Lagrange type interpolation operator which is used in the supercloseness analysis. In Section 4, we present the error estimation for supercloseness. Section 5 is devoted to the construction of the PPR operator for WG solutions. In Section 6, we present the superconvergence analysis for PPR scheme. In Section 7, we provide some numerical experiments.

\section{The WG method}
In this paper, we consider the following second-order elliptic
problem with homogeneous Dirichlet boundary condition as a model problem:
\begin{equation}\label{model}
\begin{aligned}
-\Delta u&=f,~\quad {\rm in}~\Omega,
\\
u&=0,~\quad {\rm on}~\partial \Omega,
\end{aligned}
\end{equation}
where $\Omega\subset\mathbb R^2$ is an open rectangular domain or a union of rectangular domains.

The weak formulation for (\ref{model}) can be written as: find $u\in
H_0^1(\Omega)$ such that
\begin{eqnarray}\label{weak_model}
(\nabla u,\nabla v)=(f,v),\quad \forall v\in H_0^1(\Omega),
\end{eqnarray}
where  $(\cdot,\cdot)$ is the $L^2$-inner product, and $H_0^1(\Omega)$ is a subspace of Sobolev space
$H^1(\Omega)=\{v:v\in L^2(\Omega),\nabla v\in [L^2(\Omega)]^2\}$
with vanishing boundary value.

Let $\mathcal{T}_h$ be a shape-regular rectangular mesh of domain $\Omega$.
For each element $T\in\T_h$, denote by $h_T$ the diameter of $T$.
The mesh size of $\T_h$ is defined as $h=\max_{T\in \T_h} h_T$.
Denote by ${\cal E}_h$ the set of all edges in ${\cal T}_h$ and
${\cal E}^0_h={\cal E}_h\setminus {\partial\Omega}$ the set of all
interior edges in ${\cal T}_h$. Let $Q_{k}(T)$ be a set of
polynomials that the degrees of $x$ and $y$ are no more than $k$, and let
$$
Q_{k}=\{v: v|_T\in Q_{k}(T),~\forall T\in\mathcal{T}_h\}.
$$
Define the space of weak functions on every element $T$ by
\begin{eqnarray*}
\mathcal{V}(T)&=&\{v=\{v_0,v_b\}:v_0\in L^2(T),v_b\in L^2(\partial
T)\}.
\end{eqnarray*}
Note that $v_0$ and $v_b$ are completely independent.
\begin{definition}\cite{WangYe13}
Denote by $\nabla_w v$ the weak gradient of $v\in \mathcal{V}(T)$
as a linear functional of the Sobolev space $H(div;T)=\{\bq\in[L^2(T)]^2: \nabla\cdot \bq\in L^2(T)\}$.
That is the action on any function $\bq\in H(div;T)$ is given by
\begin{eqnarray*}
\langle\nabla_w  v,\bq\rangle_T:&=&-(v_0,\nabla \cdot \bq)_T+\langle
v_b,  \bq\cdot\bn \rangle_{\partial T},
\end{eqnarray*}
where $\bn$ is the unit outward normal vector on $\partial T$.
\end{definition}

Next we define the space $W_r(T)$ to be
\begin{eqnarray*}
W_r(T)&=&[Q_{r-1,r},Q_{r,r-1}]^t,
\end{eqnarray*}
where $Q_{i,j}$ is a set of polynomials whose degrees of $x$ and $y$ are no more than $i$ and $j$, respectively.
\begin{definition}
The discrete weak gradient operator of $v\in \mathcal{V}(T)$, denoted by
$\nabla_{w,r,T} v\in W_r(T)$, is the unique function in $W_r(T)$,
satisfying
\begin{equation}\label{weak_grad}
(\nabla_{w,r,T} v,\bq)_T=-(v_0,\nabla \cdot \bq)_T+\langle v_b,
\bq\cdot\bn \rangle_{\partial T},\quad \forall \bq\in W_r(T),
\end{equation}
where $\bn$ is the unit outward normal vector on $\partial T$.
\end{definition}

Let $V_h$ and $W_h$ be the global WG spaces of weak functions and weak gradients as follows
\begin{eqnarray*}
V_h&=&\{ v=\{ v_0, v_b\}: v_0|_T\in Q_k(T), v_b|_{e}\in
P_k(e),e\subset\partial T, T\in\mathcal{T}_h\},
\\
W_h&=&\{\bq: \bq|_T\in W_k(T), T\in\mathcal{T}_h\}.
\end{eqnarray*}
Note that any weak function $v$ in $V_h$ has a single-valued component $v_b$ on
each edge $e\in{\cal E}_h$. Let $V_h^0$ be the subspace of $V_h$
with vanishing boundary value on $\partial \Omega$.

For each $v\in V_h$, the discrete weak gradient $\nabla_{w,k}v\in
W_h$ is computed piecewisely using (\ref{weak_grad}) on each element $T\in
\mathcal{T}_h$, i.e.,
\begin{align*}
(\nabla_{w,k} v)|_T=&\nabla_{w,k,T} (v|_T), \quad \forall ~v\in V_h.
\end{align*}
For simplicity, we drop the subscript $k$ from the
notation $\nabla_{w,k}$ in the rest of the paper.

Define the following bilinear forms
\begin{eqnarray}
\label{bi_form_1}s(w,v)&=&\sum_{T\in\mathcal{T}_h}h^{-\alpha}\langle w_0- w_b, v_0-
v_b\rangle_{\partial T}, \quad\alpha\geqslant1, \forall w, v\in V_h,
\\
\label{bi_form_2}a_s(w,v)&=&(\nabla_w w,\nabla_w v)_h+s(w,v), ~~~~\forall w, v\in V_h,
\end{eqnarray}
where $(\cdot,\cdot)_h=\sum_{T\in\mathcal{T}_h}(\cdot,\cdot)_T$.
\begin{lemma}
The functional $\3bar \cdot\3bar: V_h\to \mathbb{R}$ defined by
\begin{eqnarray}\label{energy norm}
\3barv\3bar^2=a_s(v,v),\quad \forall v\in V_h,
\end{eqnarray}
is a norm on the space $V_h^0$. Moreover,
\begin{eqnarray}
\sum_{T\in\mathcal{T}_h}\|\nabla v_0\|^2_T&\leq& C\3barv\3bar^2,~~\forall v\in V_h,\label{prop_norm_1}
\\
\sum_{T\in\mathcal{T}_h}h_T^{-1}\|v_0-v_b\|^2_{\partial T}&\leq&C \3barv\3bar^2,~~\forall v\in V_h.\label{prop_norm_2}
\end{eqnarray}
\end{lemma}
\begin{proof}
It is easy to see that $\3bar\cdot\3bar$ is a semi-norm in $V_h^0$.
Hence, it suffices to show that $v=0$ whenever $\3bar v\3bar=0$.
Using \eqref{bi_form_1} and \eqref{bi_form_2} we have
\begin{equation*}
  0=\3bar v\3bar^2 = a_s(v,v) = (\nabla_w v,\nabla_w
v)_h+\sum_{T\in\mathcal{T}_h}h^{-\alpha}\langle v_0- v_b, v_0-
v_b\rangle_{\partial T}.
\end{equation*}
That is $\nabla_w v={\bf 0}$ on each $T\in\mathcal{T}_h$, and
$v_0|_{e}=v_b$ on each $e\in{\cal E}_h$. It follows from
$v_0|_{e}=v_b$ that
\begin{eqnarray}
0&=&(\nabla_w v,\bq)_T
=-(v_0,\nabla\cdot\bq)_T+\langle v_b,\bq\cdot\bn\rangle_{\partial
T}\nonumber
\\
&=&(\nabla v_0,\bq)_T-\langle v_0-v_b,\bq\cdot\bn\rangle_{\partial
T}=(\nabla v_0,\bq)_T,\label{norm_help_1}
\end{eqnarray}
for any $\bq\in W_k(T)$ and $\bn$ is the outward normal of
$\partial T$. Thus $\nabla v_0={\bf 0}$ on each $T\in\mathcal{T}_h$, and
$v_0$ is a constant on each $T$. Together with $v_0|_{e}=v_b$, we
conclude that $v$ is a constant on the global domain $\Omega$. The fact $v\in
V_h^0$ implies $v=0$. As a result, $\3bar\cdot\3bar$ is a norm in space
$V_h^0$.

For any $v=\{v_0,v_b\}\in V_h$, it follows from the definition of
weak gradient, the trace inequality, the inverse inequality, and the assumption $\alpha \geq 1$ that
\begin{eqnarray*}
 &&\sum_{T\in\mathcal{T}_h}\|\nabla v_0\|^2_T  =\sum_{T\in\mathcal{T}_h}(\nabla v_0,\nabla v_0)_T = \sum_{T\in\mathcal{T}_h}(\nabla_w v,\nabla
v_0)_T+\sum_{T\in\mathcal{T}_h}\langle v_0-v_b,\nabla
v_0\cdot\bn\rangle_{\partial T} \nonumber \\
   &\leq&  \left(\sum_{T\in\mathcal{T}_h}\|\nabla_w
v\|_T^2\right)^{\frac12}\left(\sum_{T\in\mathcal{T}_h}\|\nabla
v_0\|_{T}^2\right)^{\frac12}+\left(\sum_{T\in\mathcal{T}_h}h_T^{-\alpha}\|v_0-v_b\|^2_{\partial
T}\right)^{\frac12}\left(\sum_{T\in\mathcal{T}_h}h^{\alpha}\|\nabla
v_0\|^2_{\partial T}\right)^{\frac12}\nonumber \\
   &\leq& C\3barv\3bar\left(\sum_{T\in\mathcal{T}_h}\|\nabla v_0\|^2_T\right)^{\frac{1}{2}}.
\end{eqnarray*}
We obtain (\ref{prop_norm_1}). The inequality (\ref{prop_norm_2}) follows from that $h$ is small and $\alpha\ge 1$.
\end{proof}

We consider the following weak Galerkin method: find $u_h\in V_h^0$ such that
\begin{eqnarray}\label{wg_algorithm}
a_s(u_h,v)=(f,v_0), \quad\forall v\in V_h^0,
\end{eqnarray}
where $(f,v_0)=\sum_{T\in\mathcal{T}_h}(f,v_0)_T$.

\begin{remark}
The difference between the WG method \eqref{wg_algorithm} and the classical WG method in \cite{MuWangYe15} is that the stabilizer contains a fine-tune parameter $\alpha$. Later on, it will be shown that the parameter $\alpha$ plays an important role in the supercloseness analysis in Section 4. Numerical experiments in Section 7 also demonstrate this feature.
\end{remark}

\section{Interpolation operator}
This section introduces an interpolation operator that will be used later in the superconvergence analysis.

Let $-1=\zeta_0<\zeta_1<...<\zeta_k=1$ be $k+1$ Lobatto
points on the reference interval $\hat{e}=[-1,1]$, which are $k+1$ zeros
of the Lobatto polynomial $\omega_{k+1}$. We define a
Lagrange interpolation operator $\mathcal{I}: C^0(\hat e) \to P_k(\hat e)$ such that
\begin{eqnarray}\label{inter_opera_1d}
\mathcal{I} u(x)=\sum_{i=0}^{k}u(\zeta_i)l_i(x),~~~u\in C^0(\hat{e}),
\end{eqnarray}
where $l_i$, $i=0,1,\cdots,k$ are the Lagrange interpolation associated with
Lobatto points $\zeta_i$. The following properties of $l_i$ can be easily verified:
\begin{eqnarray}
\label{proper1}&&l_i(\zeta_j)=\delta_{ij},~~~i,j=0,1,...,k,
\\
\label{proper2}&&\sum_{i=0}^{k}l_i(x)=1,~~~\forall x\in\hat{e},
\\
\label{proper3}&&\sum_{i=0}^{k}(\zeta_i-x)^ml_i(x)=0,~~~1\leq m\leq
k.
\end{eqnarray}
We also recall an interpolation error representation in \cite{ZY14}.
\begin{lemma}\label{interlemma}
Let $u\in H^{k+2}(\hat{e})$, we have the following error equation
\begin{eqnarray*}
u(x)-\mathcal{I} u(x)=C\omega_{k+1}(x)u^{(k+1)}(x)+R(u,x),
\end{eqnarray*}
where C is a constant, $\omega_{k+1}$ is the Lobatto polynomial with
order $k+1$, and
\begin{eqnarray*}
R(u,x)=\sum_{i=0}^{k}l_i(x)\int^x_{\zeta_i}\frac{(\zeta_i-t)^{k+1}}{(k+1)!}u^{(k+2)}(t)dt
.
\end{eqnarray*}
\end{lemma}

As shown in Lemma \ref{interlemma}, the interpolation
operator $\mathcal{I}$ preserves polynomials of degree up to $k$.
We composite the interpolation operators (\ref{inter_opera_1d}) in $x$- and $y$- directions to obtain an interpolation operator in the two dimensional domain $\mathcal{I}_h: C^0(\Omega)\rightarrow \mathcal{S}_h:= Q_k\cap C^0(\Omega)$ such that
\begin{equation}\label{interplation}
(\mathcal{I}_h u)|_T=\mathcal{I}_1\mathcal{I}_2 u|_T = \mathcal{I}_1\left(\sum_{i=0}^{k} u(x,\zeta^2_i)l_i(y)\right)
=\sum_{j=0}^{k}\sum_{i=0}^{k} u(\zeta^1_j,\zeta^2_i)l_i(y)l_j(x),
\end{equation}
where $\mathcal{I}_1,\mathcal{I}_2$ are the interpolation operators in
$x$-, $y$- directions, respectively. From \eqref{interplation}, it is easy to
prove $\mathcal{I}_h u\in C^0(\Omega)$. By Lemma \ref{interlemma} we have
the following estimates.
\begin{lemma}\cite{ZY14}
There exists a constant $C$ such that for any $u\in H^{k+2}(\Omega)$, the following inequality holds true
\begin{eqnarray}\label{estimates_inter}
(\nabla(u-\mathcal{I}_h u),\nabla v)\leq Ch^{k+1}|u|_{k+2}|v|_1,\quad
\forall v\in Q_k.
\end{eqnarray}
\end{lemma}
The definition of $\nabla_w$ given in (\ref{weak_grad}) and the fact that $\mathcal{I}_hv\in C^0$ yield the following lemma.
\begin{lemma}
The interpolation operator defined in (\ref{interplation}) satisfies
\begin{eqnarray}\label{continuous_prop}
(\nabla_w \mathcal{I}_hv ,\bq)_h=(\nabla \mathcal{I}_hv,\bq)_h,\quad \forall~ v\in C^0(\Omega),~\bq\in W_h,
\end{eqnarray}
where $(\cdot,\cdot)_h=\sum_{T\in\mathcal{T}_h}(\cdot,\cdot)_T$.
\end{lemma}

\section{Analysis of supercloseness}
In this section, we derive an error estimate for $\3bar \mathcal{I}_h
u-u_h\3bar$, where $u_h$ is the solution of the WG method (\ref{wg_algorithm}) and $\mathcal{I}_hu$ is the interpolation of the exact solution of problem
(\ref{model}).

\begin{theorem}\label{order_est}
Let $u\in H^{k+2}(\Omega)$ be the solution of (\ref{model}), and $u_h\in V_h$ be the solution of WG method  (\ref{wg_algorithm}).
The following error estimate holds
\begin{equation}\label{superclose}
\3bar \mathcal{I}_h u-u_h\3bar\leq Ch^{\min\{k+1,k+\frac{\alpha-1}{2}\}}\|u\|_{k+2}.
\end{equation}
\end{theorem}
\begin{proof}
Since $Q_k\subset V_h$, then $\3bar \mathcal{I}_h u-u_h\3bar$ is well-defined.
Multiplying both sides of (\ref{model}) by $v_0$, and using integration by parts, we have
\begin{equation}\label{prop_u}
\begin{aligned}
(f,v_0)=&\sum_{T\in\mathcal{T}_h}(-\Delta u,v_0)_T
=\sum_{T\in\mathcal{T}_h}(\nabla u,\nabla
v_0)_T-\sum_{T\in\mathcal{T}_h}\langle\nabla u\cdot\bn,
v_0\rangle_{\partial T}
\\
=&\sum_{T\in\mathcal{T}_h}(\nabla u,\nabla
v_0)_T-\sum_{T\in\mathcal{T}_h}\langle\nabla u\cdot\bn,
v_0-v_b\rangle_{\partial T}.
\end{aligned}
\end{equation}
Here we use the facts that the normal component $\nabla u\cdot \bn$ of the flux is continuous on all interior edges
and $v_b|_{\partial\Omega}=0$.

From (\ref{wg_algorithm}), (\ref{continuous_prop}), (\ref{prop_u}),
the Cauchy-Schwarz inequality, (\ref{estimates_inter}),
(\ref{prop_norm_1}), the property of interpolation operator
$\mathcal{I}_h$, and $\alpha\geq1$ we obtain
\begin{eqnarray*}
&&\3bar\mathcal{I}_h u-u_h\3bar^2=a_s(\mathcal{I}_h u-u_h,\mathcal{I}_h u-u_h)
\\
&=&a_s(\mathcal{I}_hu,\mathcal{I}_hu-u_h)-a_s(u_h,\mathcal{I}_hu-u_h)
\\
&=&\sum_{T\in\mathcal{T}_h}(\nabla_w\mathcal{I}_h u,\nabla_w (\mathcal{I}_h
u-u_h))_T-(f,\mathcal{I}_hu-u_0)
\\
&=&\sum_{T\in\mathcal{T}_h}(\nabla \mathcal{I}_h u,\nabla_w (\mathcal{I}_h
u-u_h))_T-\sum_{T\in\mathcal{T}_h}(\nabla u,\nabla(\mathcal{I}_hu-u_0))_T
\\
&&+\sum_{T\in\mathcal{T}_h}\langle\nabla u\cdot \bn,\mathcal{I}_h u -u_0-(\mathcal{I}_h
u-u_b)\rangle_{\partial T}
\\
&=&\sum_{T\in\mathcal{T}_h}(\nabla (\mathcal{I}_h u-u),\nabla (\mathcal{I}_h u-u_0))_T
\\
&&-\sum_{T\in\mathcal{T}_h}\langle\nabla (\mathcal{I}_h u-u)\cdot \bn,\mathcal{I}_h u -u_0-(\mathcal{I}_h
u-u_b)\rangle_{\partial T}
\\
&\leq&\sum_{T\in\mathcal{T}_h}(\nabla(\mathcal{I}_h u-u),\nabla (\mathcal{I}_h u-u_0))_T
\\
&&+\Big(\sum_{T\in\mathcal{T}_h}h_T^{\alpha}\|\nabla(\mathcal{I}_h
u-u)\|^2_{\partial
T}\Big)^{\frac12}\Big(\sum_{T\in\mathcal{T}_h}h_T^{-\alpha}\|\mathcal{I}_h u-u_0-(\mathcal{I}_h
u-u_b)\|^2_{\partial T}\Big)^{\frac12}
\\
&\leq&Ch^{\min\{k+1,k+\frac{\alpha-1}{2}\}}\|u\|_{k+2}\3bar \mathcal{I}_h
u-u_h\3bar.
\end{eqnarray*}
Here, we have used the fact that $\nabla (\mathcal{I}_hu)\in W_h$.
\end{proof}

\begin{remark}
The estimate \eqref{superclose} shows that the WG solution $u_h$ is superclose to the interpolation $\mathcal{I}_hu$ when $\alpha > 1$. It reaches the maximum rate of convergence when $\alpha = 3$. Further increasing the value of $\alpha$ will not improve the rate of convergence.
\end{remark}

\section{PPR for WG solutions}
 In this section, we introduce a gradient recovery operator $G_h$ onto space $S_h\times S_h$, with $S_h:=\{v\in C^0(\Omega): v|_T\in P_k(T),T\in\mathcal{T}_h\}$, on the rectangular mesh $\mathcal{T}_h$.
 For a WG solution $u_h$ in (\ref{wg_algorithm}), we define $G_hu_h$ on the following three types of mesh nodes  \cite{ZN05}: vertex, edge node, and internal node, see Fig. \ref{fig:three_nodes}.
\begin{figure}[!htp]
\begin{center}
\subfigure[Vertex]{
\begin{tikzpicture}
    \path (0,0) coordinate (A1);
    \path (0,1) coordinate (A2);
    \path (0,2) coordinate (A3);
    \path (1,0) coordinate (A4);
    \path (2,0) coordinate (A5);
    \path (1,1) coordinate (A6);
    \draw (A2) -- (A1) -- (A4);
    \draw (A2) -- (A6) -- (A4);
    \filldraw[black] (A1) circle(0.05);
    \filldraw[black] (A2) circle(0.05);
    \filldraw[black] (A4) circle(0.05);
    \filldraw[black] (A6) circle(0.05);
\end{tikzpicture}
}
~~~~~
\subfigure[Edge node]{
\begin{tikzpicture}
    \path (0,0) coordinate (A1);
    \path (0,1) coordinate (A2);
    \path (0,2) coordinate (A3);
    \path (1,0) coordinate (A4);
    \path (2,0) coordinate (A5);
    \path (1,1) coordinate (A6);
    \path (0,0.5) coordinate (A7);
    \path (0.5,0) coordinate (A8);
    \path (0.5,1) coordinate (A9);
    \path (1,0.5) coordinate (A10);
    \draw (A2) -- (A1) -- (A4);
    \draw (A2) -- (A6) -- (A4);
    \filldraw[black] (A7) circle(0.05);
    \filldraw[black] (A8) circle(0.05);
    \filldraw[black] (A9) circle(0.05);
    \filldraw[black] (A10) circle(0.05);
\end{tikzpicture}
}
~~~~~
\subfigure[Internal node]{
\begin{tikzpicture}
    \path (0,0) coordinate (A1);
    \path (0,1) coordinate (A2);
    \path (0,2) coordinate (A3);
    \path (1,0) coordinate (A4);
    \path (2.3,0) coordinate (A5);
    \path (1,1) coordinate (A6);
    \path (0.5,0.5) coordinate (A7);
    \draw (A2) -- (A1) -- (A4);
    \draw (A2) -- (A6) -- (A4);
    \filldraw[black] (A7) circle(0.05);
\end{tikzpicture}
}
\end{center}
\caption{Three types of nodes.}
\label{fig:three_nodes}
\end{figure}
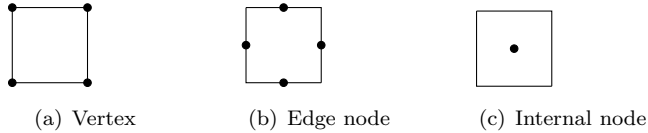

\subsection{Vertex patch}
We define a patch $K_{z}$ for every vertex $z$ by
 \begin{equation*}
 K_{z}=\{T\in\mathcal{T}_h: \bar T\cap\{z\}\neq\emptyset\}
\end{equation*}
be the union of the elements in the first layer around $z$. There can be two types of vertices. The first type is the interior vertex $z\in\Omega$, and the other one is the boundary vertex $z\in\partial \Omega$, see Fig. \ref{fig:two_vertiese} for an illustration.
  \begin{figure}[!htp]
\begin{center}
\subfigure[Interior vertex]{
\begin{tikzpicture}
    \path (0,0) coordinate (A1);
    \path (0,1) coordinate (A2);
    \path (0,2) coordinate (A3);
    \path (1,0) coordinate (A4);
    \path (2.3,0) coordinate (A5);
    \path (1,1) coordinate (A6);
    \path (1.2,1.2) coordinate (A7);
    \draw (A3) -- (A1) -- (A5);
    \draw (A2) -- (A6) -- (A4);
    \filldraw[black] (A6) circle(0.05);
    \draw node[right] at (A7) {$\Omega$};
\end{tikzpicture}
}$\qquad$$\qquad$
\subfigure[Boundary vertex]{
\begin{tikzpicture}
    \path (0,0) coordinate (A1);
    \path (0,1) coordinate (A2);
    \path (0,2) coordinate (A3);
    \path (1,0) coordinate (A4);
    \path (2,0) coordinate (A5);
    \path (1,1) coordinate (A6);
    \path (2.7,0) coordinate (A7);
    \path (1.2,1.2) coordinate (A8);
    \draw (A3) -- (A1) -- (A5);
    \draw (A2) -- (A6) -- (A4);
    \draw (A7);
    \filldraw[black] (A2) circle(0.05);
    \filldraw[black] (A1) circle(0.05);
    \filldraw[black] (A4) circle(0.05);
    \draw node[right] at (A8) {$\Omega$};
\end{tikzpicture}
}
\end{center}
\caption{Two kinds of vertices.}
\label{fig:two_vertiese}
\end{figure}
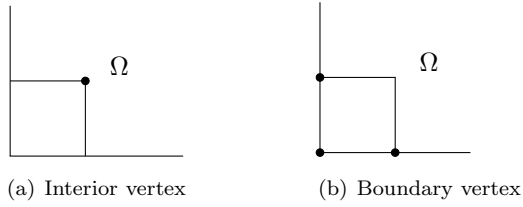

Before we introduce the PPR scheme, we need to clarify some notations.
\begin{itemize}
\item $\mathcal{N}$: All nodes in $\bar{\Omega}$. They could be vertices, edge nodes, or internal  nodes.
\item $\mathcal{N}(T)$: All mesh nodes in $\bar{T}$.
\item $\mathcal{N}_i$: $\mathcal{N}_i=\{z_{i,j}\}_{j=1}^{n_{z_i}}$ is the set of all mesh nodes in $\overline{K_{z_i}}$. Here, $n_{z_i}$ is the number of the nodes. For the linear element all nodes are vertices. For quadratic and higher-order elements, there are vertices, edge nodes, and internal nodes.
\item $\mathcal{M}^0$: All interior vertices in $\Omega$.
\item $\mathcal{M}^0(T)$: All interior vertices in $\bar{T}\cap\Omega$.
\item $\mathcal{M}^0_i$: $\mathcal{M}_i^0=\{z_{i,j}\}_{j=1}^{m_{z_i}}$ is the set of all interior vertices in $\overline{K_{z_i}}$. Denoted by $m_{z_i}$ the number of nodes in $\mathcal{M}_i^0$.
\end{itemize}

\subsection{The reformulated value $\bar{u}_h$}
 In order to obtain the recovered gradient $G_hu_h(z_i)$, we need to use the values of $u_h$ at mesh nodes in $\mathcal{N}_i$ to get an approximation $p_{k+1}\in P_{k+1}(K_{z_i})$ in the least-square sense. However, on a vertex or an edge node, the WG solution $u_h$ may have more than one value, as illustrated in Fig. \ref{fig:edge_node}. As a result, we must redefine the value of $u_h$ at those nodes.

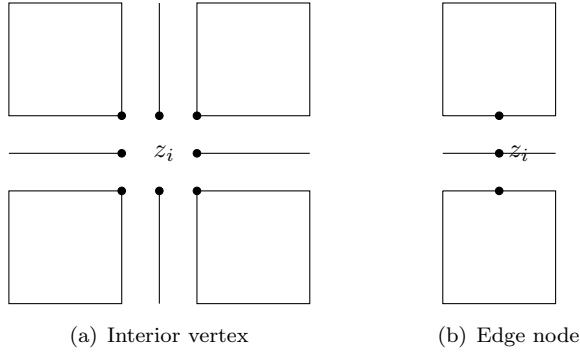
\begin{figure}[!htp]
\begin{center}
\subfigure[Interior vertex]{
\begin{tikzpicture}
    \path (0,0.5) coordinate (A11);
    \path (0,2) coordinate (A12);
    \path (0.5,0) coordinate (A21);
    \path (2,0) coordinate (A22);
    \path (0,-0.5) coordinate (A31);
    \path (0,-2) coordinate (A32);
    \path (-0.5,0) coordinate (A41);
    \path (-2,0) coordinate (A42);
    \draw (A11) -- (A12);
    \draw (A21) -- (A22);
    \draw (A31) -- (A32);
    \draw (A41) -- (A42);
    \path (0.5,0.5) coordinate (B11);
    \path (0.5,2) coordinate (B12);
    \path (2,2) coordinate (B13);
    \path (2,0.5) coordinate (B14);

    \path (-0.5,0.5) coordinate (B21);
    \path (-0.5,2) coordinate (B22);
    \path (-2,2) coordinate (B23);
    \path (-2,0.5) coordinate (B24);

    \path (-0.5,-0.5) coordinate (B31);
    \path (-2,-0.5) coordinate (B32);
    \path (-2,-2) coordinate (B33);
    \path (-0.5,-2) coordinate (B34);

    \path (0.5,-0.5) coordinate (B41);
    \path (2,-0.5) coordinate (B42);
    \path (2,-2) coordinate (B43);
    \path (0.5,-2) coordinate (B44);

    \draw (B11) -- (B12) -- (B13) -- (B14) -- (B11);
    \draw (B21) -- (B22) -- (B23) -- (B24) -- (B21);
    \draw (B31) -- (B32) -- (B33) -- (B34) -- (B31);
    \draw (B41) -- (B42) -- (B43) -- (B44) -- (B41);
    \path (-0.2,0) coordinate (a);
    \filldraw[black] (A11) circle(0.05);
    \filldraw[black] (A21) circle(0.05);
    \filldraw[black] (A31) circle(0.05);
    \filldraw[black] (A41) circle(0.05);
    \filldraw[black] (B11) circle(0.05);
    \filldraw[black] (B21) circle(0.05);
    \filldraw[black] (B31) circle(0.05);
    \filldraw[black] (B41) circle(0.05);
    \draw node[right] at (a) {$z_{i}$};
\end{tikzpicture}
}$\qquad\qquad$
\subfigure[Edge node]{
\begin{tikzpicture}
    \path (0.5,0) coordinate (A11);
    \path (2,0) coordinate (A12);
    \draw (A11) -- (A12);
    \path (0.5,0.5) coordinate (B11);
    \path (0.5,2) coordinate (B12);
    \path (2,2) coordinate (B13);
    \path (2,0.5) coordinate (B14);

    \path (0.5,-0.5) coordinate (B41);
    \path (2,-0.5) coordinate (B42);
    \path (2,-2) coordinate (B43);
    \path (0.5,-2) coordinate (B44);

    \draw (B11) -- (B12) -- (B13) -- (B14) -- (B11);
    \draw (B41) -- (B42) -- (B43) -- (B44) -- (B41);
    \path (1.25,0) coordinate (a);
     \path (1.25,0.5) coordinate (a1);
      \path (1.25,-0.5) coordinate (a2);

      \path (2.25,0) coordinate (a3);
      \draw (a3);
    \filldraw[black] (a) circle(0.05);
    \filldraw[black] (a1) circle(0.05);
    \filldraw[black] (a2) circle(0.05);
    \draw node[right] at (a) {$z_{i}$};
\end{tikzpicture}
}
\end{center}
\caption{The distribution of $u_h$ on a mesh node $z_{i}$.}
\label{fig:edge_node}
\end{figure}

 For any node $z_i\in \mathcal{N}$, denote by $\{{u}_h^j(z_i)\}_{j=1}^{l_{z_i}}$ the possible values for $u_h$ at $z_i$ where $l_{z_i}$ is the number of these values. Note that $u_h^j(z_i)$ might be the value of the interior part $u_0$ or the boundary part $u_b$ of the weak function $u_h=\{u_0,u_b\}$ at point $z_i$. We define a function $\bar{u}_h$ such that the value of $\bar{u}_h$ at $z_i$ is given by
 \begin{eqnarray}\label{refor_uh}
 \bar{u}_h(z_i)=\sum_{j=1}^{l_{z_i}}\alpha_j{u}_h^j(z_i),\quad \alpha_j\geq0,\quad \sum_{j=1}^{l_{z_i}}\alpha_j=1.
 \end{eqnarray}
Moreover, we require $\bar{u}_h\in \mathcal{S}_h$ to be a function satisfying
 \begin{eqnarray}\label{req_baru_1}
 \bar{u}_h=\sum_{z_i\in\mathcal{N}}\bar{u}_h(z_i)l_i,
 \end{eqnarray}
 where $l_i$ is the Lagrange basis associated with $z_i$.
It can be proved that the function $\bar{u}_h$ satisfies the following lemma.

\begin{lemma}\label{jump_comb}
Given $u_h=\{u_0,u_b\}\in V_h$, let $\bar{u}_h$ be defined as (\ref{refor_uh})-(\ref{req_baru_1}).
 Assume that $z_i\in\mathcal{M}^0$ is an interior vertex, $K_{z_i}$ is the patch for $z_i$, and $\mathcal{N}_i=\{z_{i,j}\}_{j=1}^{n_{z_i}}$ is the set of the nodes in $\overline{K_{z_{i}}}$, where $n_{z_i}$ is the number of the elements in $\mathcal{N}_i$. Then for $T\subset K_{z_i}$, $z_{i,j}\in \bar{T}$,  the following properties hold.
\begin{enumerate}[(i)]
\item $(\bar{u}_h-u_0)|_{T}(z_{i,j})$ can be written as the jump of $u_h$ at $z_{i,j}$, if $z_{i,j}\in\mathcal{N}_i$ is a vertex or an edge node on $\partial T$,
\item$(\bar{u}_h-u_0)|_{T}(z_{i,j})=0$, if $z_{i,j}\in\mathcal{N}_i$ is an internal mesh node in $T$.
\end{enumerate}
\end{lemma}
\begin{proof}
Without loss of generality, we consider an interior vertex $z_{i,1}$.
Assume that $u^1_0,...,u^{4}_0$, $u^{5}_b,...,u^{8}_b$ are the values of $u_h$ at $z_{i,1}$, see the left plot in Fig. \ref{fig:edge_node}. Let $\bar{u}_h(z_{i,1})=\sum_{s=1}^{4}\alpha_su^s_0+\sum_{t=5}^{8}\alpha_tu^t_b$ and $u_0|_T(z_{i,1})=u^1_0$. Then, we have
\begin{eqnarray*}
(\bar{u}_h-u_0)|_{T}(z_{i,1})=&&\sum_{s=1}^{4}\alpha_s(u^s_0-u^1_0)+\sum_{t=5}^{8}\alpha_t(u^t_b-u^1_0).
\end{eqnarray*}
This shows that $(\bar{u}_h-u_0)|_{T}(z_{i,1})$ consists of the jump of $u_h$ at $z_{i,1}$. Furthermore, it can be written as $u_0|_e(z_{i,1})-u_b|_e(z_{i,1})$ where $u_0$ and $u_b$ share the edge $e$ and $z_{i,1}$ lies on the edge $e$.

For boundary vertices and edge nodes, the proof is similar. For internal nodes, the property $(ii)$ follows directly from the definition of $\bar{u}_h$.
\end{proof}

\subsection{The PPR operator $G_h$}\label{PPR-operator}
Recall that the function $\bar{u}_h$ is defined to have a unifed value at each node. Therefore we can apply PPR scheme to construct the gradient recovery operator $G_h$. We consider the following four cases.

{\bf Case 1.} For each interior vertex $z_i\in \mathcal{M}^0$, we fit a polynomial in $P_{k+1}(K_{z_i})$ to the redefined WG solution $\bar{u}_h(z_{i,j}),j=1,...,n_{z_i}$  by the least-square method. Let $(x,y)$ be the local coordinates with respect to the origin $z_i$. The fitting polynomial is defined as
 \begin{eqnarray}\label{p_k+1}
 p_{k+1}(x,y;z_i)=\bP\ba=\hat{\bP}\hat{\ba},
 \end{eqnarray}
 where
\begin{eqnarray*}
&&\bP=\Big(1,x,y,...,x^{k+1},x^ky,...,y^{k+1}\Big),
\\
&&\hat{\bP}=\Big(1,\hat{x},\hat{y},...,\hat{x}^{k+1},\hat{x}^k\hat{y},...,\hat{y}^{k+1}\Big),
\\
&&\ba=(a_1,a_2,...,a_m)^t,~~~~~~\hat{\ba}=(a_1,ha_2,...,h^{k+1}a_m)^t,
\end{eqnarray*}
with $\hat{x}=x/h$ and $\hat{y}=y/h$, and $m=(k+2)(k+3)/2$ is the number of the basis of $P_{k+1}(K_{z_i})$. By the least-square method, the vector $\hat{\ba}$ can be solved from
\begin{eqnarray}\label{hat_a}
A^tA\hat{\ba}=A^t\bb_h,
\end{eqnarray}
where $\bb_h=(\bar{u}_h(z_{i,1}),\bar{u}_h(z_{i,2}),...,\bar{u}_h(z_{i,n_{z_i}}))^t$ and
\begin{eqnarray*}
A=\left(\begin{matrix}
     1 & \hat{x}_1 & \hat{y}_1  & ... & \hat{x}_1^{k+1} & \hat{x}_1^{k}\hat{y}_1 & ... & \hat{y}_1^{k+1} \\
     1 & \hat{x}_2 & \hat{y}_2  & ... & \hat{x}_2^{k+1}  & \hat{x}_2^{k}\hat{y}_2 & ... & \hat{y}_2^{k+1} \\
     \vdots & \vdots & \vdots & \vdots & \vdots & \vdots & \vdots & \vdots \\
     1 & \hat{x}_{n_{z_i}} & \hat{y}_{n_{z_i}} & ... & \hat{x}_{n_{z_i}}^{k+1} & \hat{x}_{n_{z_i}}^{k}\hat{y}_{n_{z_i}} & ... & \hat{y}_{n_{z_i}}^{k+1}
\end{matrix}\right)
\end{eqnarray*}
 where $(\hat{x}_j,\hat{y}_j)$ is the coordinates of $z_{i,j}$ in the reference domain.
 Define
  $G_hu_h$ at the point $z_i$ as
  \begin{eqnarray*}
  G_hu_h(z_i)=\nabla p_{k+1}(0,0;z_i).
  \end{eqnarray*}

  {\bf Case 2.} For a boundary vertex $z_i\in \partial \Omega$, we define
\begin{eqnarray*}
G_hu_h(z_i)=\frac{\sum_{z_{i,j}\in\mathcal{M}^0_i}\nabla p_{k+1}(x_j,y_j;z_{i,j})}{m_{z_i}},
\end{eqnarray*}
where $m_{z_i}$ is the number of interior vertices in $\mathcal{M}^0_i$ and $(x_j,y_j)$ is the local coordinates of $z_i$ with $z_{i,j}$ be the origin.

 {\bf Case 3.} For an edge node $z_i$ which lies on an edge between vertices $z_{i,1}$ and $z_{i,2}$, we define
   \begin{eqnarray*}
  G_hu_h(z_i)=\alpha\nabla p_{k+1}(x_1,y_1;z_{i,1})+(1-\alpha)\nabla p_{k+1}(x_2,y_2;z_{i,2}),~~0\leq\alpha\leq 1,
  \end{eqnarray*}
  where $(x_1,y_1)$ and $(x_2,y_2)$ are the coordinates of $z_{i}$ with respect to the origins $z_{i,1}$ and $z_{i,2}$, respectively.
 The weight $\alpha$ is determined by the ratio of the distances of $z_i$ to $z_{i,1}$ and $z_{i,2}$, that is $\alpha=|z_i-z_{i,2}|/|z_{i,1}-z_{i,2}|$, see Fig. \ref{fig:internal_nodes} (a).

  {\bf Case 4.} For an internal node $z_i$ which lies in an element formed by vertices $z_{i,1}$, $z_{i,2}$,..., $z_{i,4}$, we define
   \begin{eqnarray*}
  G_hu_h(z_i)=\sum_{j=1}^{4}\alpha_j\nabla p_{k+1}(x_j,y_j;z_{i,j}),~~~\sum_{j=1}^{4}\alpha_j=1,~\alpha_j\geq 0,
  \end{eqnarray*}
  where $(x_j,y_j)$ is the local coordinates of $z_{i}$ with respect to the origin $z_{i,j}$.
 The weight $\alpha_j$ is determined by the space ratio of the opposite patch to $z_{i,j}$, that is $\alpha_j=|S_j|/S$, and $S=\sum_{l=1}^4|S_l|$, see Fig. \ref{fig:internal_nodes} (b).

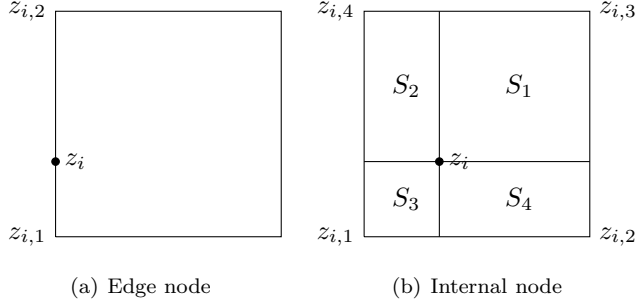
\begin{figure}[!htp]
\begin{center}
\subfigure[Edge node]{
\begin{tikzpicture}
    \path (-1,-1) coordinate (A1);
    \path (0,-1) coordinate (A2);
    \path (2,-1) coordinate (A3);
    \path (-1,0) coordinate (A4);
    \path (-1,0) coordinate (A5);
    \path (2,0) coordinate (A6);
    \path (-1,2) coordinate (A7);
    \path (0,2) coordinate (A8);
    \path (2,2) coordinate (A9);
    \path (-0.75,-0.5) coordinate (A10);
    \path (0.75,-0.5) coordinate (A11);
    \path (0.75,1) coordinate (A12);
    \path (-0.75,1) coordinate (A13);
    \draw (A1) -- (A3) -- (A9) -- (A7) -- (A1);
    \draw node[right] at (A5) {$z_{i}$};
    \draw node[left] at (A1) {$z_{i,1}$};
    \draw node[left] at (A7) {$z_{i,2}$};
    \filldraw[black] (A5) circle(0.05);
\end{tikzpicture}
}
\subfigure[Internal node]{
\begin{tikzpicture}
    \path (-1,-1) coordinate (A1);
    \path (0,-1) coordinate (A2);
    \path (2,-1) coordinate (A3);
    \path (-1,0) coordinate (A4);
    \path (0,0) coordinate (A5);
    \path (2,0) coordinate (A6);
    \path (-1,2) coordinate (A7);
    \path (0,2) coordinate (A8);
    \path (2,2) coordinate (A9);
    \path (-0.75,-0.5) coordinate (A10);
    \path (0.75,-0.5) coordinate (A11);
    \path (0.75,1) coordinate (A12);
    \path (-0.75,1) coordinate (A13);
    \draw (A1) -- (A3) -- (A9) -- (A7) -- (A1);
    \draw (A2) -- (A8);
    \draw (A4) -- (A6);
    \draw node[right] at (A5) {$z_{i}$};
    \draw node[left] at (A1) {$z_{i,1}$};
    \draw node[right] at (A3) {$z_{i,2}$};
    \draw node[right] at (A9) {$z_{i,3}$};
    \draw node[left] at (A7) {$z_{i,4}$};
    \draw node[right] at (A10) {$S_3$};
    \draw node[right] at (A11) {$S_4$};
    \draw node[right] at (A12) {$S_1$};
    \draw node[right] at (A13) {$S_2$};
    \filldraw[black] (A5) circle(0.05);
\end{tikzpicture}
}
\end{center}
\caption{The lengths/areas that distributed by node $z_i$.}
\label{fig:internal_nodes}
\end{figure}

\begin{remark}
For any $u_h\in V_h$, $G_hu_h$ is defined as the linear combination of the values of $G_hu_h$ at the interior vertex.
For $u\in C^0(\Omega)$, we define $G_hu$ by
\begin{eqnarray}\label{G_hI_h}
G_hu=G_h\mathcal{I}_hu,
\end{eqnarray}
where $\mathcal{I}_h:C^0(\Omega)\rightarrow \mathcal{S}_h\subset V_h$ is the interpolation operator given in (\ref{interplation}).
\end{remark}

\section{Superconvergence estimates}
In this section, we report several properties of the operator $G_h$, and analyze the superconvergence  between $\nabla u$ and $G_hu_h$.

The following lemma can be directly verified following the same procedure as Lemma 3.10 in \cite{Zhang2005}.
\begin{lemma}\label{lem3.10}
Let $z_i\in\mathcal{M}^0$ be an interior vertex with the patch $K_{z_i}$, and let $p_{k+1}(\cdot,\cdot;{{z}}_i)$ be the least square polynomial of the function $v\in \mathcal{S}_h$ in the patch $K_{z_i}$. Then there is a constant $C$ such that
\begin{equation*}
|\nabla p_{k+1}(\cdot,\cdot;{{z}}_i)|_{\infty,K_{{{z}}_i}}\leq Ch^{-1}|v|_{1,K_{{{z}}_i}}.
\end{equation*}
\end{lemma}
By the definition given in subsection \ref{PPR-operator}, $G_h$ is a polynomial-preserving operator which satisfies the following lemma.
\begin{lemma}
The gradient recovery operator $G_h$ satisfies the following properties
\begin{eqnarray}
&&G_hu_h=G_h\bar{u}_h,~~~~\forall u_h\in V_h,\label{inq_Gh_3}
\\
&&\|\nabla u-G_hu\|\leq Ch^{k+1}|u|_{k+2},~~~~\forall u\in H^{k+2}(\Omega),\label{inq_Gh_2}
\end{eqnarray}
where $C$ is a constant and $\bar{u}_h\in \mathcal{S}_h$ satisfying (\ref{refor_uh})-(\ref{req_baru_1}) is the redefined function of $u_h$.
\end{lemma}

The following lemma provides an important tool in establishing our main result.
\begin{lemma}
For $u_h\in V_h$, the following property holds true
\begin{eqnarray}
\3baru_h\3bar^2 \geq\|G_h\bar{u}_h\|^2,\label{inq_Gh_4}
\end{eqnarray}
where $\bar{u}_h\in \mathcal{S}_h$ satisfying (\ref{refor_uh})-(\ref{req_baru_1}) is the redefined function of $u_h$.
\end{lemma}

\begin{proof}
We will prove \eqref{inq_Gh_4} in three steps.

{\bf Step 1}. For any $T\in\mathcal{T}_h$, recall that ${\mathcal{M}^0}(T)$ denotes the set of the interior vertices in $\bar T\cap\Omega$. Then, from the definition of $G_h$, we have
\begin{eqnarray*}
\|G_h\bar{u}_h\|_{0,T}\leq C|T|^\frac{1}{2}\|G_h\bar{u}_h\|_{\infty,T}\leq C|T|^\frac{1}{2}\max_{{z}_i\in{\mathcal{M}^0(T)}}\{|\nabla p_{k+1}(\cdot,\cdot;{{z}}_i)|_{\infty,K_{{{z}}_i}}\}.
\end{eqnarray*}
Using Lemma \ref{lem3.10}, we have
\begin{eqnarray*}
\|G_h\bar{u}_h\|_{0,T}\leq C |T|^\frac{1}{2}\max_{{{z}}_i\in{\mathcal{M}^0(T)}}\{h^{-1}|\bar{u}_h|_{1,K_{{{z}}_i}}\} \leq C \max_{{{z}}_i\in{\mathcal{M}^0(T)}}\{|\bar{u}_h|_{1,K_{{{z}}_i}}\}.
\end{eqnarray*}
It follows that
\begin{eqnarray}\label{main_inq_1}
\|G_h\bar{u}_h\|^2=\sum_{T\in\mathcal{T}_h}\|G_h\bar{u}_h\|^2_{0,T}\leq C \sum_{z_i\in\mathcal{M}^0}|\bar{u}_h|^2_{1,K_{z_i}}.
\end{eqnarray}

{\bf Step 2}. Define the auxiliary function $\tilde{u}_h$ as
\begin{eqnarray*}
\tilde{u}_h=\bar{u}_h-u_0.
\end{eqnarray*}
For any interior vertex $z_i\in\mathcal{M}^0$, it follows from the definition of $\bar{u}_h$ and $u_0$ that $\tilde{u}_h$ is a piecewise polynomial on $K_{z_i}$.
Then from the triangle inequality  we have
\begin{eqnarray*}
|\bar{u}_h|^2_{1,K_{z_i}}=|\tilde{u}_h+u_0|^2_{1,K_{z_i}}
\leq|\tilde{u}_h|^2_{1,K_{z_i}}+|u_0|^2_{1,K_{z_i}}.
\end{eqnarray*}
It follows from (\ref{prop_norm_1}) that
\begin{eqnarray}
\sum_{z_i\in\mathcal{M}^0}|\bar{u}_h|^2_{1,K_{z_i}}\leq&&\sum_{z_i\in\mathcal{M}^0}(|\tilde{u}_h|^2_{1,K_{z_i}}+|u_0|^2_{1,K_{z_i}})
\leq C(\sum_{z_i\in\mathcal{M}^0}|\tilde{u}_h|^2_{1,K_{z_i}}+\3baru_h\3bar^2). \label{eq step2}~~~~
\end{eqnarray}

{\bf Step 3}. We shall prove
\begin{eqnarray}
|\tilde{u}_h|^2_{1,K_{z_i}}\leq C\3baru_h\3bar^2_{K_{z_i}}. \label{eq step3}
\end{eqnarray}
First, we consider an element $T_1\subset K_{z_i}$. Let $\tilde{u}_h|_{T_1}=\sum_{{z}_{i,j}\in\mathcal{N}(T_1)}{\tilde{u}_h}({z_{i,j}})l_{{i,j}}$, where $l_{i,j}(z_{k,l})=\delta_{i,k}\delta_{j,l}$ are the Lagrange bases. Let $\hat{l}_{{i,j}}$ be the affine function for ${l}_{{i,j}}$ on the reference domain. Note that $|\nabla\hat{l}_{{i,j}}|$ is uniformly bounded, then we obtain
\begin{eqnarray*}
|{\tilde{u}_h}|^2_{1,T_1}=&&\int_{T_1}|\nabla {\tilde{u}_h}|^2dx
=\int_{T_1}\Big|\nabla \Big(\sum_{z_{i,j}\in\mathcal{N}(T_1)}{\tilde{u}_h}({z_{i,j}})l_{{i,j}}\Big)\Big|^2dx
\\
\leq&&C\sum_{z_{i,j}\in\mathcal{N}(T_1)}|{\tilde{u}_h}({z_{i,j}})|^2\int_{\hat{T}_1}|\nabla\hat{l}_{{i,j}}|^2d{\hat x}
~\leq C\sum_{z_{i,j}\in\mathcal{N}(T_1)}|{\tilde{u}_h}({z_{i,j}})|^2.
\end{eqnarray*}
Let $\mathcal{E}(T_1)=\{e\in\mathcal{E}_h : e\cap \mathcal{N}(T_1)\neq\emptyset\}$ and $[u_h]_e$ be the jump of $u_h$ over $e$.
From Lemma \ref{jump_comb}, we know that the values of ${\tilde{u}_h}$ on the mesh nodes on $\partial T_1$ is the combination of the jump of $u_h$ on edges $e\in\mathcal{E}(T_1)$, the values of ${\tilde{u}_h}$ on the internal mesh nodes in $T_1$ are zeros. Using the inverse inequality $\|v\|_{\infty,e}\leq Ch^{-\frac{1}{2}}\|v\|_{0,e}$ and (\ref{prop_norm_2}), we obtain
\begin{eqnarray*}
\sum_{{z}_{i,j}\in\mathcal{N}(T_1)}|{\tilde{u}_h}({z_{i,j}})|^2 &&
\leq C\sum_{e\in\mathcal{E}(T_1)}|[u_h]_e|^2_{\infty,e}
\leq C\sum_{e\in\mathcal{E}(T_1)}h^{-1}|[u_h]_e|^2_{0,e}
\\
&&\leq C\sum_{T\in \mathcal{T}_h(T_1)}h^{-1}\langle u_0-u_b,u_0-u_b\rangle_{\partial T}
\leq  C\sum_{T\in \mathcal{T}_h(T_1)}\3baru_h\3bar^2_T,
\end{eqnarray*}
where $\mathcal{T}_h(T_1):=\{T\in\mathcal{T}_h: T\cap e\neq\emptyset, e\in\mathcal{E}(T_1)\}$.
For other three elements $T\in K_{z_i}$, the proof can be finished similarly.
Finally, combining \eqref{main_inq_1}, \eqref{eq step2}, and \eqref{eq step3}, we have
\begin{eqnarray*}
\|G_h\bar{u}_h\|^2\leq&& C \sum_{z_i\in\mathcal{M}^0}|\bar{u}_h|^2_{1,K_{z_i}}
\leq C (\sum_{z_i\in\mathcal{M}^0}|{\tilde{u}_h}|^2_{1,K_{z_i}}+\3baru_h\3bar^2)
\leq C\3baru_h\3bar^2.
\end{eqnarray*}
\end{proof}

Now we are ready to state our main result for superconvergence.
\begin{theorem}
Let $u\in H^{k+2}(\Omega)$ be the solution of (\ref{model}) and $u_h\in V_h$ be the solution of (\ref{wg_algorithm}). Let $G_h u_h$ be the recovered gradient by PPR introduced in Section 5.3. Then we have the following error estimate
\begin{eqnarray}
\|G_hu_h-\nabla u\|\leq C h^{\min\{k+1,k+\frac{\alpha-1}{2}\}}\|u\|_{k+2}. \label{main thm}
\end{eqnarray}
\end{theorem}
\begin{proof}
It follows from (\ref{inq_Gh_3}), (\ref{G_hI_h}), (\ref{inq_Gh_2}), (\ref{inq_Gh_4}), and (\ref{superclose}) that
\begin{eqnarray*}
&&\|G_hu_h-\nabla u\|^2
\\
\leq&&\|G_hu_h-G_h\mathcal{I}_hu\|^2+\|G_h\mathcal{I}_hu-\nabla u\|^2
\\
\leq&&\|G_h(\bar{u}_h-\mathcal{I}_hu)\|^2+Ch^{2(k+1)}|u|^2_{k+2}
\\
\leq&&\3baru_h-\mathcal{I}_hu\3bar^2+Ch^{2(k+1)}|u|^2_{k+2}
\\
\leq&&Ch^{2({\min\{k+1,k+\frac{\alpha-1}{2}\}})}\|u\|^2_{k+2},
\end{eqnarray*}
which completes the proof.
\end{proof}

\begin{remark}
The estimate \eqref{main thm} shows that the gradient recovery $G_hu_h$ is superconvergent to $\nabla u$ when $\alpha \geq 1$. As the value of $\alpha$ increases, the convergence rate will also increase, and it reach the maximum rate of convergence $k+1$ when $\alpha =3$.
\end{remark}

\section{Numerical experiments}
In this section, we present some numerical examples to demonstrate the convergence of WG methods and the PPR recovery. We test our algorithm for the $Q_1$ and $Q_2$ elements, and choose different stabilizing parameters in our numerical algorithms for comparison. We focus on $\3bar u_h-\mathcal{I}_h u\3bar$, the error between the WG solution and its Lagrange interpolation in the energy norm, and $\|G_hu_h-\nabla u\|$, the error between the PPR recovered gradient and the true gradient in the $L^2$ norm.

\begin{example}
(\textbf{Convergence for $k=1$ on uniform meshes})
\end{example}
In this example, we consider the problem (\ref{model}) in the unit square $(0,1)\times (0,1)$, and use a family of uniform Cartesian meshes.
The weak Galerkin space is given by
\begin{equation}
V_h=\{ v=\{ v_0, v_b\}: v_0|_T\in Q_1(T), v_b|_{e}\in
P_1(e),e\subset\partial T,T\in\mathcal{T}_h\}.
\end{equation}
The discrete weak gradient $\nabla_w v$ on each element
$T\in\mathcal{T}_h$ is defined as
\begin{equation}
(\nabla_{w} v,\bq)_T=-(v_0,\nabla\cdot\bq)_T+\langle v_b,  \bq\cdot\bn\rangle_{\partial
T},\quad \forall \bq\in [Q_{0,1},Q_{1,0}]^t.
\end{equation}
The right-hand side function $f$ is chosen such that the exact
solution is
\begin{eqnarray}
u=\sin(\pi x)\sin(\pi y). \label{sine example}
\end{eqnarray}
Tables \ref{sinsink=1uni} and \ref{sinsink=1uniPPR} report the convergence rates of $\3bar u_h-\mathcal{I}_h u\3bar$ and  $\|G_hu_h-\nabla u\|$, respectively. Different values of the stabilizing parameter $\alpha$ have been tested. Here the parameter $N = 1/h$ denotes the number of rectangles in each direction. Table \ref{sinsink=1uni} clearly demonstrates that the convergence rate is $\min\{k+1,k+\frac{\alpha-1}{2}\}$, which is consistent with the error estimates \eqref{superclose}. Table \ref{sinsink=1uniPPR}  indicates the superconvergence behavior of the PPR recovery. We note that for $\alpha=1,2$, the numerical results seem to be even better than our theoretical analysis \eqref{main thm}.

\begin{table}[!htb]
\caption{Example 7.1. Convergence of $\3baru_h-\mathcal{I}_hu\3bar$ for $k=1$ on uniform meshes.}
\label{sinsink=1uni} \centering
\begin{tabular}{|c|c|c|c|c|c|c|}
\hline
\multirow{2}{*}{$N$} & \multicolumn{2}{|c|}{$\alpha=1$} & \multicolumn{2}{|c|}{$\alpha=2$} & \multicolumn{2}{|c|}{$\alpha=3$} \\
\cline{2-7}
& $\3baru_h-\mathcal{I}_h u\3bar$ & order & $\3baru_h-\mathcal{I}_h u\3bar$ & order & $\3baru_h-\mathcal{I}_h u\3bar$ & order\\
\hline
8 & 7.3081e-01 & -- & 3.0840e-01 & -- & 1.3216e-01 & --\\
\hline
16 & 3.6645e-01 & 0.9959 & 1.0916e-01 & 1.4983 & 3.3156e-02 & 1.9949\\
\hline
32 & 1.8335e-01 & 0.9990 & 3.8584e-02 & 1.5004 & 8.2964e-03 & 1.9987\\
\hline
64 & 9.1690e-02 & 0.9998 & 1.3637e-02 & 1.5005 & 2.0746e-03 & 1.9997\\
\hline
128 & 4.5847e-02 & 0.9999 & 4.8204e-03 & 1.5003 & 5.1867e-04 & 1.9999\\
\hline
\end{tabular}
\end{table}

\begin{table}[!htb]
\caption{Example 7.1. Convergence of  $\|G_hu_h-\nabla u\|$ for $k=1$ on uniform meshes.} \label{sinsink=1uniPPR} \centering
\begin{tabular}{|c|c|c|c|c|c|c|}
\hline
\multirow{2}{*}{$N$} & \multicolumn{2}{|c|}{$\alpha=1$} & \multicolumn{2}{|c|}{$\alpha=2$} & \multicolumn{2}{|c|}{$\alpha=3$} \\
\cline{2-7}
& $\|G_h u_h-\nabla u\|$ & order & $\|G_h u_h-\nabla u\|$ & order & $\|G_h u_h-\nabla u\|$ & order\\
\hline
8 & 1.0250e-01 & -- & 1.4942e-01 & -- & 1.5950e-01 & --\\
\hline
16 & 2.1339e-02 & 2.2641 & 4.2838e-02 & 1.8024 & 4.4857e-02 & 1.8302\\
\hline
32 & 5.1285e-03 & 2.0569 & 1.1614e-02 & 1.8831 & 1.1909e-02 & 1.9133\\
\hline
64 & 1.2692e-03 & 2.0146 & 3.0197e-03 & 1.9433 & 3.0591e-03 & 1.9608\\
\hline
128 & 3.1624e-04 & 2.0049 & 7.6942e-04 & 1.9726 & 7.7448e-04 & 1.9818\\
\hline
\end{tabular}
\end{table}

 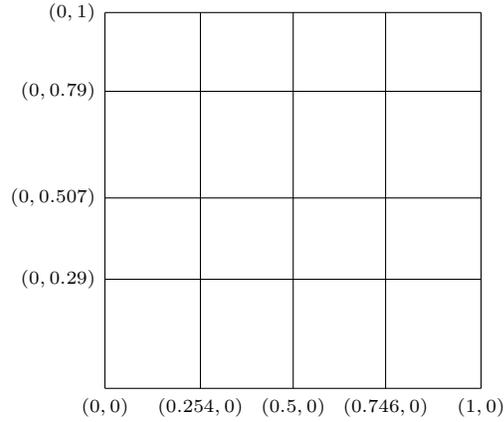
\begin{figure}[!htp]
\begin{center}
\begin{tikzpicture}
    \path (0,0) coordinate (A1);
    \path (0,1.451) coordinate (A2);
    \path (0,2.533) coordinate (A3);
    \path (0,3.952) coordinate (A4);
    \path (0,5) coordinate (A5);
    \path (5,0) coordinate (B1);
    \path (5,1.451) coordinate (B2);
    \path (5,2.533) coordinate (B3);
    \path (5,3.952) coordinate (B4);
    \path (5,5) coordinate (B5);
    \path (1.269,0) coordinate (C1);
    \path (2.5025,0) coordinate (C2);
    \path (3.731,0) coordinate (C3);
    \path (1.269,5) coordinate (D1);
    \path (2.5025,5) coordinate (D2);
    \path (3.731,5) coordinate (D3);
    \draw (A1) -- (A5);
    \draw (B1) -- (B5);
    \draw (C1) -- (D1);
    \draw (C2) -- (D2);
    \draw (C3) -- (D3);
    \draw (A5) -- (B5);
    \draw (A4) -- (B4);
    \draw (A3) -- (B3);
    \draw (A2) -- (B2);
    \draw (A1) -- (B1);
    \draw node[below] at (A1) {\scriptsize${ (0,0)}$};
    \draw node[below] at (C1) {\scriptsize${ (0.254,0)}$};
    \draw node[below] at (C2) {\scriptsize ${(0.5,0)}$};
    \draw node[below] at (C3) {\scriptsize ${(0.746,0)}$};
    \draw node[below] at (B1) {\scriptsize ${(1,0)}$};
    \draw node[left] at (A2) {\scriptsize ${(0,0.29)}$};
    \draw node[left] at (A3) {\scriptsize ${(0,0.507)}$};
    \draw node[left] at (A4) {\scriptsize ${(0,0.79)}$};
    \draw node[left] at (A5) {\scriptsize ${(0,1)}$};
\end{tikzpicture}
\end{center}
\caption{The initial partition.}
\label{fig:partition}
\end{figure}

\begin{example}
(\textbf{Convergence for $k=1$ on heterogeneous meshes})
\end{example}
In this example, we investigate the superconvergence behavior on heterogeneous rectangular meshes.
We use the same function \eqref{sine example} in this test. The initial mesh is randomly perturbed from the uniform mesh, and is given in Fig. \ref{fig:partition}.
The subsequent meshes are produced by uniform refinement.
Errors are reported in Tables \ref{sinsink=1dis} and \ref{sinsink=1disPPR}, in which similar superconvergence phenomenon is observed on these quasi-uniform rectangular meshes.
Although the convergence rate of PPR recovered gradient for $\alpha=1$ is not as high as in the uniform mesh in Example 7.1, but it is still higher than the analytical result \eqref{main thm}.

\begin{table}[htb!]
\caption{Example 7.2. Convergence of $\3baru_h-\mathcal{I}_hu\3bar$ for $k=1$ on heterogeneous meshes.} \label{sinsink=1dis} \centering
\begin{tabular}{|c|c|c|c|c|c|c|}
\hline
\multirow{2}{*}{$N$} & \multicolumn{2}{|c|}{$\alpha=1$} & \multicolumn{2}{|c|}{$\alpha=2$} & \multicolumn{2}{|c|}{$\alpha=3$} \\
\cline{2-7}
& $\3baru_h-\mathcal{I}_h u\3bar$ & order & $\3baru_h-\mathcal{I}_h u\3bar$ & order & $\3baru_h-\mathcal{I}_h u\3bar$ & order\\
\hline
8 & 7.3711e-01 & -- & 3.1389e-01 & -- & 1.3595e-01 & --\\
\hline
16 & 3.6979e-01 & 0.9952 & 1.1114e-01 & 1.4979 & 3.4117e-02 & 1.9945\\
\hline
32 & 1.8504e-01 & 0.9988 & 3.9284e-02 & 1.5003 & 8.5374e-03 & 1.9986\\
\hline
64 & 9.2540e-02 & 0.9997 & 1.3885e-02 & 1.5005 & 2.1349e-03 & 1.9997\\
\hline
128 & 4.6272e-02 & 0.9999 & 4.9079e-03 & 1.5003 & 5.3375e-04 & 1.9999\\
\hline
\end{tabular}
\end{table}
\begin{table}[H]
\caption{Example 7.2. Convergence of $\|G_hu_h-\nabla u\|$ for $k=1$ on heterogeneous meshes.} \label{sinsink=1disPPR} \centering
\begin{tabular}{|c|c|c|c|c|c|c|}
\hline
\multirow{2}{*}{$N$} & \multicolumn{2}{|c|}{$\alpha=1$} & \multicolumn{2}{|c|}{$\alpha=2$} & \multicolumn{2}{|c|}{$\alpha=3$} \\
\cline{2-7}
& $\|G_h u_h-\nabla u\|$ & order & $\|G_h u_h-\nabla u\|$ & order & $\|G_h u_h-\nabla u\|$ & order\\
\hline
8 & 1.1753e-01 & -- & 1.5431e-01 & -- & 1.6310e-01 & --\\
\hline
16 & 2.8433e-02 & 2.0474 & 4.4400e-02 & 1.7972 & 4.6371e-02 & 1.8145\\
\hline
32 & 8.2896e-03 & 1.7782 & 1.2080e-02 & 1.8779 & 1.2381e-02 & 1.9051\\
\hline
64 & 2.6255e-03 & 1.6587 & 3.1488e-03 & 1.9398 & 3.1897e-03 & 1.9566\\
\hline
128 & 8.7219e-04 & 1.5899 & 8.0340e-04 & 1.9706 & 8.0870e-04 & 1.9797\\
\hline
\end{tabular}
\end{table}

\begin{example}
(\textbf{Convergence for $k=2$})
\end{example}
In this example, we test the superconvergence properties for some higher order WG approximations. In particular, we choose $k=2$.
Tables \ref{sinsink=2uni} and \ref{sinsink=2uniPPR} list errors and the convergence rates for
$u_h-\mathcal{I}_hu$ and $G_hu_h-\nabla u$, respectively.

\begin{table}[H]
\caption{Example 7.3. Convergence of $\3baru_h-\mathcal{I}_hu\3bar$ for $k=2$ on uniform meshes.} \label{sinsink=2uni} \centering
\begin{tabular}{|c|c|c|c|c|c|c|}
\hline
\multirow{2}{*}{$N$} & \multicolumn{2}{|c|}{$\alpha=1$} & \multicolumn{2}{|c|}{$\alpha=2$} & \multicolumn{2}{|c|}{$\alpha=3$} \\
\cline{2-7}
& $\3baru_h-\mathcal{I}_h u\3bar$ & order & $\3baru_h-\mathcal{I}_h u\3bar$ & order & $\3baru_h-\mathcal{I}_h u\3bar$ & order\\
\hline
8 & 4.7148e-02 & -- & 2.0112e-02 & -- & 8.4797e-03 & --\\
\hline
16 & 1.1947e-02 & 1.9805 & 3.5666e-03 & 2.4954 & 1.0609e-03 & 2.9987\\
\hline
32 & 2.9972e-03 & 1.9950 & 6.3078e-04 & 2.4994 & 1.3263e-04 & 2.9998\\
\hline
64 & 7.4996e-04 & 1.9987 & 1.1151e-04 & 2.4999 & 1.6580e-05 & 3.0000\\
\hline
128 & 1.8753e-04 & 1.9997 & 1.9713e-05 & 2.5000 & 2.0725e-06 & 3.0000\\
\hline
\end{tabular}
\end{table}
\begin{table}[H]
\caption{Example 7.3. Convergence of $\|G_hu_h-\nabla u\|$ for $k=2$ on uniform meshes.} \label{sinsink=2uniPPR} \centering
\begin{tabular}{|c|c|c|c|c|c|c|}
\hline
\multirow{2}{*}{$N$} & \multicolumn{2}{|c|}{$\alpha=1$} & \multicolumn{2}{|c|}{$\alpha=2$} & \multicolumn{2}{|c|}{$\alpha=3$} \\
\cline{2-7}
& $\|G_h u_h-\nabla u\|$ & order & $\|G_h u_h-\nabla u\|$ & order & $\|G_h u_h-\nabla u\|$ & order\\
\hline
8 & 4.4523e-02 & -- & 1.4339e-02 & -- & 1.0521e-02 & --\\
\hline
16 & 7.4403e-03 & 2.5811 & 1.2958e-03 & 3.4680 & 9.8726e-04 & 3.4137\\
\hline
32 & 1.2791e-03 & 2.5402 & 1.0979e-04 & 3.5610 & 8.6269e-05 & 3.5165\\
\hline
64 & 2.2415e-04 & 2.5126 & 9.4297e-06 & 3.5414 & 7.6241e-06 & 3.5002\\
\hline
128 & 3.9531e-05 & 2.5034 & 8.3447e-07 & 3.4983 & 6.9244e-07 & 3.4608\\
\hline
\end{tabular}
\end{table}

Data in Tables \ref{sinsink=2uniPPR} demonstrate that the PPR gradient recovery is superconvergent to $\nabla u$. Numerical experiments for all three choices of $\alpha$ are of higher order convergence than our theoretical results. This surprising observation somehow indicates that there might be a more subtle relationship between the PPR recovery for WG solution and exact solution. As for now, \eqref{main thm} is the best theoretical estimate we can achieve. Improving the theoretical estimate will be an interesting future research project.

\begin{example}
{\textbf{(Convergence for less smooth functions)}}
\end{example}
In this example, we consider the problem $-\Delta u=1$, on the unit square with the homogeneous Dirichlet boundary condition. The exact solution can be written as
\begin{eqnarray}
u(x,y)=&&\frac{x(1-x)+y(1-y)}{4}-\frac{2}{\pi^3}\sum_{i=0}^{\infty}\frac{1}{(2i+1)^3(1+e^{-(2i+1)\pi})}\nonumber
\\
&&\cdot\Big[(e^{-(2i+1)\pi y} +e^{-(2i+1)\pi (1-y)})\sin(2i+1)\pi x\nonumber
\\
&&+(e^{-(2i+1)\pi x} +e^{-(2i+1)\pi (1-x)})\sin(2i+1)\pi y\Big]. \label{nonsmooth function}
\end{eqnarray}
The solution \eqref{nonsmooth function} is not as smooth as functions in previous examples.
In fact, the function is in $H^{3-\epsilon}(\Omega)$ for any $\epsilon > 0$, but not in $H^3(\Omega)$,
and it has a weak singularity $r^2\ln r$ at the four corners of the domain. Obviously, this nonsmoothness will affect the superconvergence of our numerical schemes.

In the numerical test below, we truncate first fifty terms of the infinite sum as an approximation of the exact solution. We test both $k=1$ and $k=2$ cases on uniform meshes. Tables \ref{multik=1uni} and \ref{multik=1uniPPR} report the convergence for $\3bar u_h-\mathcal{I}_hu\3bar$ and $\|G_hu_h-\nabla u\|$ for $k=1$. Tables \ref{multik=2uni} and \ref{multik=2uniPPR} report the convergence for $k=2$.

\begin{table}[!hbt]
\caption{Example 7.4. Convergence of $\3baru_h-\mathcal{I}_hu\3bar$ for $k=1$ on uniform meshes.} \label{multik=1uni} \centering
\begin{tabular}{|c|c|c|c|c|c|c|}
\hline
\multirow{2}{*}{$1/h$} & \multicolumn{2}{|c|}{$\alpha=1$} & \multicolumn{2}{|c|}{$\alpha=2$} & \multicolumn{2}{|c|}{$\alpha=3$} \\
\cline{2-7}
& $\3baru_h-\mathcal{I}_h u\3bar$ & order & $\3baru_h-\mathcal{I}_h u\3bar$ & order & $\3baru_h-\mathcal{I}_h u\3bar$ & order\\
\hline
8 & 8.1780e-02 & -- & 3.5022e-02 & -- & 1.4989e-02 & --\\
\hline
16 & 4.1405e-02 & 0.9820 & 1.2408e-02 & 1.4970 & 3.7598e-03 & 1.9952\\
\hline
32 & 2.0794e-02 & 0.9936 & 4.3854e-03 & 1.5005 & 9.4322e-04 & 1.9950\\
\hline
64 & 1.0412e-02 & 0.9979 & 1.5498e-03 & 1.5006 & 2.3720e-04 & 1.9915\\
\hline
128 & 5.2084e-03 & 0.9994 & 5.4783e-04 & 1.5003 & 5.9954e-05 & 1.9842\\
\hline
\end{tabular}
\end{table}
\begin{table}[H]
\caption{Example 7.4. Convergence of $\|G_hu_h-\nabla u\|$ for $k=1$ on uniform meshes.} \label{multik=1uniPPR} \centering
\begin{tabular}{|c|c|c|c|c|c|c|}
\hline
\multirow{2}{*}{$1/h$} & \multicolumn{2}{|c|}{$\alpha=1$} & \multicolumn{2}{|c|}{$\alpha=2$} & \multicolumn{2}{|c|}{$\alpha=3$} \\
\cline{2-7}
& $\|G_h u_h-\nabla u\|$ & order & $\|G_h u_h-\nabla u\|$ & order & $\|G_h u_h-\nabla u\|$ & order\\
\hline
8 & 1.5971e-02 & -- & 1.0089e-02 & -- & 9.1308e-03 & --\\
\hline
16 & 4.6605e-03 & 1.7769 & 2.4973e-03 & 2.0157 & 2.3884e-03 & 1.9347\\
\hline
32 & 1.4394e-03 & 1.6950 & 6.4013e-04 & 1.9639 & 6.2871e-04 & 1.9526\\
\hline
64 & 4.7289e-04 & 1.6059 & 1.6594e-04 & 1.9477 & 1.6476e-04 & 1.9320\\
\hline
128 & 1.6182e-04 & 1.5471 & 4.3901e-05 & 1.9183 & 4.3783e-05 & 1.9120\\
\hline
\end{tabular}
\end{table}

\begin{table}[H]
\caption{Example 7.4. Convergence of $\3baru_h-\mathcal{I}_hu\3bar$ for $k=2$ on uniform meshes.} \label{multik=2uni} \centering
\begin{tabular}{|c|c|c|c|c|c|c|}
\hline
\multirow{2}{*}{$1/h$} & \multicolumn{2}{|c|}{$\alpha=1$} & \multicolumn{2}{|c|}{$\alpha=2$} & \multicolumn{2}{|c|}{$\alpha=3$} \\
\cline{2-7}
& $\3baru_h-\mathcal{I}_h u\3bar$ & order & $\3baru_h-\mathcal{I}_h u\3bar$ & order & $\3baru_h-\mathcal{I}_h u\3bar$ & order\\
\hline
8 & 5.0628e-03 & -- & 2.3506e-03 & -- & 1.0261e-03 & --\\
\hline
16 & 1.4579e-04 & 1.7961 & 4.7629e-04 & 2.3031 & 1.4729e-04 & 2.8004\\
\hline
32 & 4.0745e-04 & 1.8392 & 9.3998e-05 & 2.3411 & 2.3929e-05 & 2.6219\\
\hline
64 & 1.1214e-04 & 1.8613 & 1.9466e-05 & 2.2717 & 7.9370e-06 & 1.5921\\
\hline
128 & 2.9974e-05 & 1.9035 & 3.6114e-06 & 2.4303 & 1.4716e-06 & 2.4313\\
\hline
\end{tabular}
\end{table}

\begin{table}[H]
\caption{Example 7.4. Convergence of $\|G_hu_h-\nabla u\|$ for $k=2$ on uniform meshes.} \label{multik=2uniPPR} \centering
\begin{tabular}{|c|c|c|c|c|c|c|}
\hline
\multirow{2}{*}{$1/h$} & \multicolumn{2}{|c|}{$\alpha=1$} & \multicolumn{2}{|c|}{$\alpha=2$} & \multicolumn{2}{|c|}{$\alpha=3$} \\
\cline{2-7}
& $\|G_h u_h-\nabla u\|$ & order & $\|G_h u_h-\nabla u\|$ & order & $\|G_h u_h-\nabla u\|$ & order\\
\hline
8 & 5.8500e-03 & -- & 2.4904e-03 & -- & 2.1339e-03 & --\\
\hline
16 & 1.5089e-03 & 1.9549 & 5.7978e-04 & 2.1028 & 5.4716e-04 & 1.9807\\
\hline
32 & 3.8325e-04 & 1.9771 & 1.4062e-04 & 2.0437 & 1.3706e-04 & 1.9905\\
\hline
64 & 9.7663e-05 & 1.9724 & 3.5701e-05 & 1.9777 & 3.5286e-05 & 1.9550\\
\hline
128 & 2.3997e-05 & 2.0250 & 3.5269e-06 & 2.0881 & 8.3497e-06 & 2.0789\\
\hline
\end{tabular}
\end{table}

We note that for $k=1$, our superconvergence analysis requires the exact solution to be in $H^3$. Data in Tables \ref{multik=1uni}-\ref{multik=1uniPPR} demonstrate that the convergence orders perfectly match or are even better than orders in our theoretical analysis. For higher order approximation $k=2$, to get the analytical superconvergence order, we need the exact solution to be in $H^4$. However, the exact solution here is barely in $H^3$. Hence, some superconvergence behavior does not exist, which is reflected in Tables \ref{multik=2uni}-\ref{multik=2uniPPR}.

{\bf A final remark}. The condition regarding $\alpha$ is sharp in the supercloseness result (4.1). As we can see from data in Tables 7.1, 7.3, 7.5, and 7.7, the convergent rate follows loyally to the predicted $k+(\alpha-1)/2$. On the other hand, the condition regarding $\alpha$ may not be necessary for our superconvergence result in Theorem 6.4 as we can see from data in Tables 7.2, 7.4, 7.6, and 7.8: when $\alpha =1, 2$, the supercloseness lost but the superconvergence still exists, since the supercloseness result (4.1) is a sufficient condition for Theorem 6.4, not a necessary condition.


\end{document}